\documentclass[a4paper,11pt,leqno]{article}

\usepackage[latin1]{inputenc}
\usepackage[T1]{fontenc} 
\usepackage[english]{babel}
\usepackage{verbatim}
\usepackage{calligra}
\usepackage{enumerate}
\usepackage{dsfont}
\usepackage{amsfonts}
\usepackage{amsmath}
\usepackage{amsthm}
\usepackage{amssymb}
\usepackage{mathtools}
\usepackage{mathrsfs}
\usepackage{color}
\usepackage{graphicx}
\usepackage{bm}

\usepackage{tikz}
\usepackage[retainorgcmds]{IEEEtrantools}

\usepackage{graphicx}		
\usepackage[hypcap=false]{caption}

\DeclareMathOperator*{\tend}{\longrightarrow}

\DeclareMathOperator*{\D}{\rm{div}}

\theoremstyle{definition}
\newtheorem{defi}{Definition}[section]

\newtheorem{rmk}[defi]{Remark}

\theoremstyle{plane}
\newtheorem{thm}[defi]{Theorem}
\newtheorem{prop}[defi]{Proposition}
\newtheorem{cor}[defi]{Corollary}
\newtheorem{lemma}[defi]{Lemma}

\newcommand{\tbf}{\textbf}

\newcommand{\tsl}{\textsl}

\newcommand{\mbb}{\mathbb}

\newcommand{\mc}{\mathcal}

\newcommand{\veps}{\varepsilon}

\newcommand{\what}{\widehat}

\newcommand{\vphi}{\varphi}

\newcommand{\ra}{\rightarrow}

\renewcommand{\k}{\kappa}

\renewcommand{\t}{\tau}

\newcommand{\de}{\delta}

\newcommand{\R}{\mathbb{R}}

\newcommand{\C}{\mathbb{C}}
\newcommand{\N}{\mathbb{N}}
\newcommand{\Z}{\mathbb{Z}}
\newcommand{\T}{\mathbb{T}}
\renewcommand{\P}{\mathbb{P}}

\renewcommand{\div}{{\rm div}\,}

\newcommand{\curl}{{\rm curl}\,}

\newcommand{\Id}{{\rm Id}\,}
\newcommand{\Supp}{{\rm Supp}\,}

\newcommand{\dt}{ \, {\rm d} t}
\newcommand{\B}{B^s_{\infty, r}}
\newcommand{\al}{\alpha}
\newcommand{\bt}{\beta}

\allowdisplaybreaks

\def\d{\partial}
\def\div{{\rm div}\,}

\begin{document}

\newcommand{\dimitri}[1]{\textcolor{red}{[***DC: #1 ***]}}
\newcommand{\fra}[1]{\textcolor{blue}{[***FF: #1 ***]}}

\title{\textsc{\Large{\textbf{Symmetry breaking in ideal magnetohydrodynamics: \\ the role of the velocity}}}}

\author{\normalsize\textsl{Dimitri Cobb}$\,^1\qquad$ and $\qquad$
\textsl{Francesco Fanelli}$\,^{2}$ \vspace{.5cm} \\
\footnotesize{$\,^{1,} \,^2\;$ \textsc{Universit\'e de Lyon, Universit\'e Claude Bernard Lyon 1}}  \vspace{.1cm} \\
{\footnotesize \it Institut Camille Jordan -- UMR 5208}  \vspace{.1cm}\\
{\footnotesize 43 blvd. du 11 novembre 1918, F-69622 Villeurbanne cedex, FRANCE} \vspace{.2cm} \\
\footnotesize{$\,^{1}\;$\ttfamily{cobb@math.univ-lyon1.fr}}, $\;$
\footnotesize{$\,^{2}\;$\ttfamily{fanelli@math.univ-lyon1.fr}}
\vspace{.2cm}
}

\date\today

\maketitle

\subsubsection*{Abstract}
{\footnotesize
The ideal magnetohydrodynamic equations are, roughly speaking, a quasi-linear symmetric hyperbolic system of PDEs, but 
not all the unknowns play the same role in this system. Indeed, in the regime of small magnetic fields, the equations are close to the incompressible Euler equations.
In the present paper, we adopt this point of view to study questions linked with the lifespan of strong solutions to the ideal magnetohydrodynamic equations. 
First of all, we prove a continuation criterion in terms of the velocity field only. Secondly, we refine the explicit lower bound for the lifespan
of $2$-D flows found in \cite{CF3}, by relaxing the regularity assumptions on the initial magnetic field.
}

\paragraph*{\small 2010 Mathematics Subject Classification:}{\footnotesize 35Q35 
(primary);
76W05, 
35B60, 
76B03 
(secondary).}

\paragraph*{\small Keywords: }{\footnotesize ideal MHD; velocity field; Els\"asser variables; blow-up criterion; lifespan.}

\section{Introduction}\label{s:intro}

In this paper, we are concerned with ideal incompressible magnetohydrodynamics (MHD for short), which are governed by the following system of equations:
\begin{equation}\label{ieq:mhd}
\begin{cases}
\partial_t u + (u \cdot \nabla)u + \nabla \left( \Pi + \dfrac{1}{2} |b|^2 \right) = (b \cdot \nabla) b\\[1ex]
\partial_t b + (u \cdot \nabla)b - (b \cdot \nabla)u = 0\\[1ex]
\D(u) = 0\,.
\end{cases}
\end{equation}
These equations describe the motion of an ideal incompressible magnetofluid, that is an inviscid, perfectly conducting and  incompressible fluid which is subject to a self-generated magnetic field.
We set these equations on the whole $d$-dimensional space $\R^d$. The vector fields $u, b: \R \times \R^d \tend \R^d$ are, respectively, the velocity and magnetic fields of the fluid,
while the scalar quantity $\Pi: \R \times \R^d \tend \R$ is the hydrodynamic pressure.

Since the 1930s, this set of equations has been the subject of intense studies by physicists and, for the past thirty years, by mathematicians. Amongst the many problems that have been explored (wave propagation, magnetofluid topology, stationary solutions and their stability,
\textsl{etc}.), the well-posedness theory has proven to be one of the most challenging, as it is still unknown whether system \eqref{ieq:mhd} possesses global solutions, even in the case of
two dimensions of space. This contrasts very much with the case of the incompressible Euler equations, with no magnetic field, where the existence and uniqueness of a global regular solution
has been proved for planar solutions $d = 2$. The presence of the magnetic field makes all the methods that work for the Euler system inoperable for ideal MHD.

In this article, we extend some results we previously obtained in \cite{CF3} regarding the lifespan of Besov-Lipschitz solutions.
In doing so, we will highlight the fact that the velocity field $u$ plays a special role in \eqref{ieq:mhd}, in that we will require different levels of regularity on $u$ and $b$ to prove a
continuation criterion and a lower bound on the lifespan of solutions.
We point out that this type of results owes to the particular nature of ideal MHD, as it reaches further than the standard theory of quasi-linear symmetric hyperbolic systems,
where all unknowns must recieve a similar treatment due to the symmetric nature of the system.

\medbreak
To begin with, let us present, in the next subsection, some generalities on the ideal MHD system \eqref{ieq:mhd}; there, we will also recall some well-posedness results obtained in \cite{CF3}. Those results
constitute the starting point of the present analysis.

\subsection{Some insights on the ideal MHD}\label{ss:els}

Roughly speaking, equations \eqref{ieq:mhd} can be viewed as a first-order quasi-linear symmetric hyperbolic system (of course, this is not completely correct, due to the presence of the pressure term).
Thus, it is natural to study well-posedness questions in a functional framework based on finite energy conditions.

However, the structure of the equations is much richer than that, as may be highlighted by a change of unknowns. Specifically, by introducing the so-called
\emph{Els\"asser variables}
\begin{equation} \label{eq:Els-var}
\al = u + b \qquad \quad \text{ and } \qquad \quad \bt = u - b\,,
\end{equation}
the ideal MHD system \eqref{ieq:mhd} can be recasted into the following system of transport equations:
\begin{equation}\label{eq:els}
\begin{cases}
\partial_t \al + (\bt \cdot \nabla) \al + \nabla \pi_1 = 0 \\[1ex]
\partial_t \bt + (\al \cdot \nabla) \bt + \nabla \pi_2 = 0\\[1ex]
\D(\al) = \D(\bt) = 0.
\end{cases}
\end{equation}
In the above, $\pi_1$ and $\pi_2$ are two possibly distinct scalar functions, which enforce the two independent divergence-free conditions $\D(\al) = 0$ and $\D (\bt) = 0$.
While it is clear that all solutions of the ideal MHD system also solve \eqref{eq:els}, with in addition $\nabla\pi_1=\nabla\pi_2=\nabla(\Pi+|b|^2/2)$,
the converse is not, in general, true without imposing some kind of condition the solutions
must satisfy at infinity $|x| \tend + \infty$. For instance, if the solution $(\al,\bt)$ of \eqref{eq:els} lies in some $L^p$ space, where $1\leq p<+\infty$,
then it can be shown that $(u,b)$, obtained inverting transformation \eqref{eq:Els-var}, solves \eqref{ieq:mhd}.
We refer to the discussion in Section 4 of \cite{CF3} for more on this issue; see also \cite{Cobb} for the statement of a sharp equivalence result.

To the best of our knowledge, the Els\"asser formulation \eqref{eq:els} of the ideal MHD equations was involved, in a way or another, in all well-posedness results for system \eqref{ieq:mhd}
obtained so far, starting from the the very first works of Schmidt \cite{Sch} and Secchi \cite{Secchi}. We refer \tsl{e.g.} to \cite{C-K-S}, \cite{C-M-Z}, \cite{BBV}, \cite{CF3}
and references therein for more recent studies.
It should be noted that equations \eqref{eq:els} are basically a system of transport equations. This makes it possible to propagate integrability assumptions other than $L^2$,
and to solve the ideal MHD system (provided the equivalence between equations \eqref{ieq:mhd} and \eqref{eq:els} holds) in spaces based on $L^p$ conditions, for any $p\in\,]1,+\infty]$.

The previous observation was used in \cite{CF3} to solve the ideal MHD system in endpoint Besov spaces $\B$ included in the space of globally Lipschitz functions.
The result is given in the next statement (see Theorems 2.1 and 2.4 in \cite{CF3}). We remark that the finite energy condition on the solutions is there, mainly to guarantee the equivalence
between the original system \eqref{ieq:mhd} and its Els\"asser formulation \eqref{eq:els}. We refer to Theorem 4.3 in \cite{CF3} for more details about that issue.

\begin{thm} \label{t_i:WP}
Let $d\geq2$. Let $(s,r)\in\R\times[1,+\infty]$ satisfy either $s>1$, or $s=r=1$. Then the ideal MHD system \eqref{ieq:mhd} is well-posed, locally in time,
in the space
$$
\mbb X^s_r\,:=\,\left\{(u,b)\in\left(\B(\R^d)\right)^2\;\Big|\quad \div(u)\,=\,\div(b)\,=\,0\qquad\mbox{ and }\qquad u\,,\,b\;\in\,L^2(\R^d)\; \right\}\,,
$$
and there exists a $T>0$ such that the flow map $t\mapsto \big(u(t),b(t)\big)$ belongs to $C^0\big([0,T];\mbb X^s_r\big)$ if $1\leq r<+\infty$, to $C^0_w\big([0,T];\mbb X^s_\infty\big)$ if $r=+\infty$.
In addition, 
if $T<+\infty$ and
\begin{equation} \label{eq:cont-Lip}
\int_0^{T} \Big( \big\| \nabla u(t) \big\|_{L^\infty} + \big\| \nabla b(t) \big\|_{L^\infty} \Big) \dt < +\infty\,,
\end{equation}
then $(u,b)$ can be continued beyond $T$ into a solution of \eqref{ieq:mhd} with the same regularity.
\end{thm}

Solving in critical spaces becomes particularly important in the case of space dimension $d=2$, because, in that setting, one can show an improved lower bound on the lifespan
of the solutions. We point out that this bound does not rely on classical quasi-linear hyperbolic theory, but is really tied to the special structure of the equations.
The precise estimate is contained in the following statement (this corresponds to Theorem 2.6 of \cite{CF3})
\begin{thm} \label{t_i:lifespan}
Consider an initial datum $\big(u_0, b_0\big)$ such that $u_0, b_0 \in L^2(\R^2)\cap B^2_{\infty,1}(\R^2)$, with $\div(u_0)=\div(b_0)=0$.
Then, the lifespan $T>0$ of the corresponding solution $(u, b)$ of the $2$-D ideal MHD problem \eqref{ieq:mhd}, given by Theorem \ref{t_i:WP},
enjoys the following lower bound:
\begin{equation*}
T\,\geq\,\frac{C}{\big\| (u_0, b_0) \big\|_{L^2 \cap B^2_{\infty, 1}}}
\log \left\{ 1 + C\,\log \left[ 1 + C\,\log \left( 1 + C\,\frac{\big\|  (u_0, b_0) \big\|_{L^2 \cap B^1_{\infty, 1}}}{\|b_0\|_{B^1_{\infty, 1}}} \right) \right] \right\},
\end{equation*}
where $C>0$ is a ``universal'' constant, independent of the initial datum.
\end{thm}

The interest for the previous statement comes from the fact that it implies an ``asymptotically global'' well-posedness result, in the following sense: if,
for some $\veps>0$, one has $\left\|b_0\right\|_{B^1_{\infty,1}}\,\sim\,\veps$, then the lifespan $T_\veps>0$ of the corresponding solution
verifies the property $T_\veps\,\longrightarrow\,+\infty$ for $\veps\ra0^+$.
This is consistent with the fact that, in the regime $\veps\ra0^+$, the ideal MHD system \eqref{ieq:mhd} reduces to the incompressible Euler equations,
which are globally well-posed in $2$-D.

A phenomenom of this type has already been observed in \cite{DF} for the non-homogeneous incompressible Euler system, where, in the regime of near constant densities,
the lifespan of the unique Besov-Lipschitz solution can be shown to tend to infinity, with an explicit lower bound for the lifespan. See also \cite{F-L} for a similar result for
a quasi-incompressible Euler system.
However, we have to remark that the lower bound of Theorem \ref{t_i:lifespan} does not hold in a critical setting (namely, at the level of $B^1_{\infty,1}$ regularity), and
requires higher smoothness assumptions for both initial data $u_0$ and $b_0$.

\medbreak
We will comment a little bit more on the contents of Theorems \ref{t_i:WP} and \ref{t_i:lifespan} in the next subsection, when presenting an overview of our main results.

\subsection{Main goals of the paper} \label{ss:goals}

The previous Theorems \ref{t_i:WP} and \ref{t_i:lifespan} deal with the velocity field and the magnetic field in a quite symmetric way
(apart, of course, in the explicit lower bound on the lifespan in the second statement, where $b_0$ plays a special role). However,
looking at the equations hints that the velocity $u$ and the magnetic field $b$ do not play exactly the same role 
in system \eqref{ieq:mhd}, despite the symmetric structure of the system: in fact,
the magnetic field equation is \emph{bilinear} in $(u,b)$.

Our main purpose here is to push forward this observation as far as we can, especially in two directions. First of all, we aim at finding a continuation criterion
in terms only of $u$. Secondly, we want to establish a lower bound on the lifespan of the solutions in dimension $d=2$, which requires additional regularity $B^2_{\infty,1}$ only on the initial
velocity field $u_0$.

We will explain better the improvements in both directions here below. Before doing this, we want to clarify that, in our analysis, we will need to resort again to the Els\"asser system
\eqref{eq:els}. Thus, in order to guarantee the equivalence between \eqref{ieq:mhd} and \eqref{eq:els}, we place ourselves in the same setting adopted in \cite{CF3}, namely we will always
work in the framework of \emph{finite energy} solutions of the ideal MHD equations.
Besides, this framework will enable us to use quite freely the Leray projection operator $\P$, and perform $\B$ estimates on the projected system. This differs from the approach
employed in \cite{CF3}, which was based on recasting the equations in the vorticity formulation.

\subsubsection{A continuation criterion based on the velocity} \label{sss:cont-crit}

Theorem \ref{t_i:WP}, and especially the continuation condition \eqref{eq:cont-Lip}, points at a classical phenomenon in the context of quasilinear symmetric hyperbolic systems:
the lifespan $T^*$ of solutions may be characterized by the finiteness of their $L^1_{T^*}(W^{1, \infty})$ norms.
In fact, as for the Beale-Kato-Majda continuation criterion \cite{B-K-M} for the Euler equations, it is possible to show (see \cite{C-K-S})
that a time $T>0$ is prior to explosion, namely $T < T^*$, if and only if
\begin{equation*}
\int_0^T \Big( \| \omega \|_{L^\infty} + \| j \|_{L^\infty} \Big) \dt < +\infty,
\end{equation*}
where $\omega=\curl(u)$ and $j=\curl(b)$ are the vorticity and electrical current matrices of the fluid.
A number of refinements to that criterion exist for ideal MHD, in different spaces. We refer \tsl{e.g.} to \cite{C-C-M} or \cite{C-M-Z} for results in that spirit. 

However, we have to remark that, 
contrary to the momentum equation, the magnetic field equation is \emph{linear} with respect to $b$. So, we may wish for a continuation criterion based on the velocity alone.
In this respect, we will prove that $T < T^*$ as long as
\begin{equation}\label{ieq:contU}
\int_0^T \left\| \nabla^2 u(t) \right\|_{L^\infty} \dt < + \infty.
\end{equation}
The fact that one needs second order derivatives of $u$ in the previous criterion is reasonable. This loss of derivatives has to be ascribed to the hyperbolic nature of the system:
bounding $\nabla b$ in $L^\infty$ requires a control on $\nabla^2u$, since estimates cannot be closed in a $L^\infty$ setting.

Similarly, we notice that the Els\"asser variables $\al = u+b$ and $\bt = u-b$ solve linear equations \eqref{eq:els} too. This makes it possible to establish a continuation criterion based on either
$\al$ or $\bt$. More precisely, we prove that $T < T^*$ as long as
\begin{equation}\label{ieq:conUpB}
\int_0^T \big\| \, \omega \pm j \, \big\|_{B^0_{\infty, 1}} \dt\,<\,+\infty\,.
\end{equation}
In fact, we will prove that $\big\| \, \omega + j \, \big\|_{L^1_T(B^0_{\infty, 1})}<+\infty $ if and only if $\big\| \, \omega - j \, \big\|_{L^1_T(B^0_{\infty, 1})}<+\infty $, so that
the $\pm$ sign in \eqref{ieq:conUpB} is not ambiguous.
Of course, this is not really surprising, as the magnetic field is a pseudovector: the equations remain unchanged when substituting $-b$ to $b$.

To conclude this part, we remark that the continuation cirterion \eqref{ieq:contU}, although it requires a control of the second derivative $\nabla^2 u$, can be formulated with the $L^\infty$ norm,
unlike \eqref{ieq:conUpB} which uses the $B^0_{\infty, 1}$ one. This is a consequence of the fact that we work at critical regularity $B^0_{\infty,1}$,
combined with a fundamental property of the magnetic field equation: it naturally preserves the
divergence-free property $\D(b) = 0$ in time. This means that no addition of a gradient term is needed to keep the magnetic field solenoidal, resulting in a simpler evolution equation.

\subsubsection{Improved lower bound for the lifespan of solutions} \label{sss:lifespan}

As already mentioned, there is no global well-posedness theory for the ideal MHD system, even in the case of two dimensions of space. However, as highlighted
by Theorem \ref{t_i:lifespan} above, we expect better behaviour from the solutions in the regime of small magnetic fields, as, in this case, the system is close to the $2$-D Euler equations.

We remark that, for Theorem \ref{t_i:lifespan} to hold, the initial data must possess at least $B^2_{\infty,1}$ regularity. Indeed,
the method of the proof (which is contained \cite{CF3}) required to find lower order estimates (namely, in $B^1_{\infty,1}$) for the magnetic field,
and use it as a measure of how close the solution $(u, b)$ is to the Euler system: if $b$ is small, then $(u, b)$ almost solves the Euler equations.
This was a problem, as the magnetic field equation involves first order derivatives of the velocity, so that, in a non-Hilbertian functional framework,
such estimates can only be based on \emph{higher order} ones for $u$, whence the regularity assumption on the initial data.

In this paper, we use a different method to relax this regularity requirement: we will only need the initial velocity field $u_0$ to be $B^2_{\infty, 1}$, while
$b_0 \in B^1_{\infty, 1}$ will suffice. To achieve this improvement, we will instead compare directly the ideal MHD system to the Euler equations, which we know have a global solution
$v$ at the level of regularity of the initial datum $u_0$. In doing this, instead of using the magnetic field to  measure the proximity with the Euler system, we introduce Els\"asser-type variables
\begin{equation*}
\delta \al = u + b - v \qquad \quad \text{and} \qquad \quad \delta \bt = u - b - v,
\end{equation*}
which have the nice property of solving a set of transport equations whose forcing terms only involve derivatives with respect to the Euler solution $v$.

\medskip

Before moving on, let us comment a bit further on the physical nature of $2$-D ideal MHD. While $2$-D Euler equations can easily be understood as simply describing a planar fluid,
that is not quite so with MHD, which has an inherent three-dimensional nature: the magnetic field circulates around the electrical current, so they cannot be simultaneously coplanar.

In $2$-D MHD, the fluid evolves in a plane, but the electrical current is always normal to the plane of the fluid, so that the magnetic field will indeed be planar.
In particular, the electrical current may always be represented as a scalar function $j = \partial_1 b_2 - \partial_2 b_1$.
This is of course analogous to the fact that planar fluids have a $3$-D vorticity that is, at all times, normal to the plane of motion.

\subsubsection*{Structure of the paper}
Before concluding this introduction, we give a short overview of the paper.

In the next section, we introduce some tools from Fourier analysis and Littlewood-Paley theory, which we will need in our analysis.
Section \ref{s:cont-crit} is devoted to the statement and proof of some new continuation criteria, as described in Paragraph \ref{sss:cont-crit} above. Finally, the improved lower bound on the
lifespan of the solutions in two space dimensions, requiring higher regularity on the initial velocity only, will be the topic of Section \ref{s:lifespan}.

\subsubsection*{Acknowledgements}

{\small
The authors wish to express their deep gratitude to Rapha\"el Danchin, whose interesting remarks about a preliminary version of their previous work \cite{CF3} motivated the present study.

The work of the second author has been partially supported by the LABEX MILYON (ANR-10-LABX-0070) of Universit\'e de Lyon, within the program ``Investissement d'Avenir''
(ANR-11-IDEX-0007),  and by the projects BORDS (ANR-16-CE40-0027-01), SingFlows (ANR-18-CE40-0027), all operated by the French National Research Agency (ANR).
Both authors have been partially supported by the project CRISIS (ANR-20-CE40-0020-01), operated by the French National Research Agency (ANR).
}

\section{Fourier analysis toolbox} \label{s:tools}

In this section, we give a summary of the harmonic analysis tools we will use throughout this article.
We start by giving the main ideas of Littlewood-Paley analysis and paradifferential calculus, and then we present their use in the theory of transport equations.
We conclude with a short section on the Leray projector, where we present some useful inequalities involving it.

If not otherwise specified, we refer to Chapter 2 of \cite{BCD} for full details on this part.

\subsection{Non-homogeneous Littlewood-Paley theory and Besov spaces} \label{ss:LP}

Here we recall the basic principles of Littlewood-Paley theory. 
We focus only on the $\R^d$ case, even though a similar analysis can be performed also in the case of the torus $\T^d$.

First of all, let us introduce a non-homogeneous dyadic partition of unity with
respect to the Fourier variable. 
We fix a smooth radial function $\chi$ supported in the ball $B(0,2)$, equal to $1$ in a neighborhood of $B(0,1)$
and such that $r\mapsto\chi(r\,e)$ is nonincreasing over $\R_+$ for all unitary vectors $e\in\R^d$. Set
$\varphi\left(\xi\right)=\chi\left(\xi\right)-\chi\left(2\xi\right)$ and
$\vphi_j(\xi):=\vphi(2^{-j}\xi)$ for all $j\geq0$.
The dyadic blocks $(\Delta_j)_{j\in\Z}$ are defined by\footnote{Throughout we agree  that  $f(D)$ stands for 
the pseudo-differential operator $u\mapsto\mc{F}^{-1}[f(\xi)\,\what u(\xi)]$.} 
$$
\Delta_j\,:=\,0\quad\mbox{ if }\; j\leq-2,\qquad\Delta_{-1}\,:=\,\chi(D)\qquad\mbox{ and }\qquad
\Delta_j\,:=\,\varphi(2^{-j}D)\quad \mbox{ if }\;  j\geq0\,.
$$
We  also introduce the following low frequency cut-off operator:
\begin{equation} \label{eq:S_j}
S_ju\,:=\,\chi(2^{-j}D)\,=\,\sum_{k\leq j-1}\Delta_{k}\qquad\mbox{ for }\qquad j\geq0\,.
\end{equation}
Note that $S_j$ is a convolution operator. More precisely, if we denote $\mc F(f)\,=\,\what f$ the Fourier transform of a function $f$ and $\mc F^{-1}$
the inverse Fourier transform, after defining
$$
K_0\,:=\,\mc F^{-1}\chi\qquad\qquad\mbox{ and }\qquad\qquad K_j(x)\,:=\,\mathcal{F}^{-1}\left[\chi (2^{-j}\,\cdot\,)\right] (x) = 2^{jd}K_0(2^j x)\,,
$$
we have, for all $j\in\N$ and all tempered distributions $u\in\mc S'$, that $S_ju\,=\,K_j\,*\,u$.
Thus the $L^1$ norm of $K_j$ is independent of $j\geq0$, hence $S_j$ maps continuously $L^p$ into itself, for any $1 \leq p \leq +\infty$.

With this preparation, the following \emph{Littlewood-Paley decomposition} of tempered distributions holds true:
\begin{equation} \label{eq:LP}
\forall\,u\in\mc S'\,,\qquad\qquad u=\sum_{j\geq -1}\Delta_ju\qquad \mbox{ in }\; \mc S'\,.
\end{equation}

Next, let us recall the so-called \emph{Bernstein inequalities}, which explain the way derivatives act on spectrally localised functions.
  \begin{lemma} \label{l:bern}
Let  $0<r<R$.   A constant $C$ exists so that, for any non-negative integer $k$, any couple $(p,q)$ 
in $[1,+\infty]^2$, with  $p\leq q$,  and any function $u\in L^p$,  we  have, for all $\lambda>0$,
$$
\displaylines{
{\Supp}\, \widehat u \subset   B(0,\lambda R)\quad
\Longrightarrow\quad
\|\nabla^k u\|_{L^q}\, \leq\,
 C^{k+1}\,\lambda^{k+d\left(\frac{1}{p}-\frac{1}{q}\right)}\,\|u\|_{L^p}\;;\cr
{\Supp}\, \widehat u \subset \{\xi\in\R^d\,|\, r\lambda\leq|\xi|\leq R\lambda\}
\quad\Longrightarrow\quad C^{-k-1}\,\lambda^k\|u\|_{L^p}\,
\leq\,
\|\nabla^k u\|_{L^p}\,
\leq\,
C^{k+1} \, \lambda^k\|u\|_{L^p}\,.
}$$
\end{lemma}

The second Bernstein inequality may be extended to Fourier multipliers whose symbol are homogeneous functions.
This is particularly useful when dealing with the Leray projection operator. 

\begin{lemma}\label{l:homMult}
Let $\sigma : \R^d \backslash \{ 0 \} \tend \mc \C$ be a smooth homogeneous function of degree $m \in \mathbb{Z}$. Then, for all $j \geq 0$ and $p \in [1, + \infty]$, we have
\begin{equation*}
\forall u \in \mc S', \qquad \| \sigma (D) \Delta_j f \|_{L^p} \leq C 2^{jm} \| \Delta_j u \|_{L^p}.
\end{equation*}
\end{lemma}

Now, by use of the Littlewood-Paley decomposition \eqref{eq:LP}, we can define the class of non-homogeneous Besov spaces.
\begin{defi} \label{d:B}
  Let $s\in\R$ and $1\leq p,r\leq+\infty$. The \emph{non-homogeneous Besov space}
$B^{s}_{p,r}\,=\,B^s_{p,r}(\R^d)$ is defined as the subset of tempered distributions $u$ for which
$$
\|u\|_{B^{s}_{p,r}}\,:=\,
\left\|\left(2^{js}\,\|\Delta_ju\|_{L^p}\right)_{j\geq -1}\right\|_{\ell^r}\,<\,+\infty\,.
$$
\end{defi}

\medskip

In this article, we mainly work with the $B^m_{\infty, 1}$ spaces ($m = 0, 1, 2$), which are embedded in usual spaces of bounded functions: we have
\begin{equation*}
B^m_{\infty, 1} \hookrightarrow W^{m, \infty}.
\end{equation*}
In particular, the space $B^1_{\infty, 1}$ is contained in the space of globally Lipschitz functions $W^{1,\infty}$.

\subsection{Non-homogeneous paradifferential calculus}\label{ss:NHPC}

In this subsection, we recall some useful results from paradifferential calculus. We mainly focus on the Bony paraproduct decomposition (after J.-M. Bony, see \cite{Bony}) and on some
basic commutator estimates.

We start by introducing the paraproduct operator. 
Formally, the product  of two tempered distributions $u$ and $v$ may be decomposed into 
\begin{equation*} 
u\,v\;=\;\mathcal{T}_u(v)\,+\,\mathcal{T}_v(u)\,+\,\mathcal{R}(u,v)\,,
\end{equation*}
where we have defined
$$
\mathcal{T}_u(v)\,:=\,\sum_jS_{j-1}u\,\Delta_j v\qquad\qquad\mbox{ and }\qquad\qquad
\mathcal{R}(u,v)\,:=\,\sum_j\sum_{|k-j|\leq1}\Delta_j u\,\Delta_{k}v\,.
$$
The above operator $\mc T$ is called ``paraproduct'' whereas
$\mc R$ is called ``remainder''.
The paraproduct and remainder operators have many nice continuity properties. 
The following ones will be of constant use in this paper.
\begin{prop}\label{p:op}
For any $(s,p,r)\in\R\times[1,+\infty]^2$ and $t>0$, the paraproduct operator 
$\mathcal{T}$ maps continuously $L^\infty\times B^s_{p,r}$ in $B^s_{p,r}$ and  $B^{-t}_{\infty,\infty}\times B^s_{p,r}$ in $B^{s-t}_{p,r}$.
Moreover, the following estimates hold:
$$
\|\mathcal{T}_u(v)\|_{B^s_{p,r}}\,\leq\, C\,\|u\|_{L^\infty}\,\|\nabla v\|_{B^{s-1}_{p,r}}\qquad\mbox{ and }\qquad
\|\mathcal{T}_u(v)\|_{B^{s-t}_{p,r}}\,\leq\, C\|u\|_{B^{-t}_{\infty,\infty}}\,\|\nabla v\|_{B^{s-1}_{p,r}}\,.
$$
For any $(s_1,p_1,r_1)$ and $(s_2,p_2,r_2)$ in $\R\times[1,+\infty]^2$ such that 
$s_1+s_2>0$, $1/p:=1/p_1+1/p_2\leq1$ and~$1/r:=1/r_1+1/r_2\leq1$,
the remainder operator $\mathcal{R}$ maps continuously~$B^{s_1}_{p_1,r_1}\times B^{s_2}_{p_2,r_2}$ into~$B^{s_1+s_2}_{p,r}$.
In the case $s_1+s_2=0$, provided $r=1$, the operator $\mathcal{R}$ is continuous from $B^{s_1}_{p_1,r_1}\times B^{s_2}_{p_2,r_2}$ with values
in $B^{0}_{p,\infty}$.
\end{prop}

The consequence of this proposition is that the spaces $\B$ are Banach algebras as long as $s > 0$.
Instead, notice that the space $B^0_{\infty, r}$ is \emph{not} an algebra.
%


Now, we switch to considering some commutator estimates. The first one is contained in the next statement (see Lemma 2.100 and Remark 2.101 in \cite{BCD}).

\begin{lemma}\label{l:CommBCD}
Assume that $v \in \B$ with $(s, r)\in\R\times[1,+\infty]$ satisfying $s> 1$, or $s=r=1$.
Denote by $\big[ v \cdot \nabla, \Delta_j \big] f\,=\,(v \cdot \nabla) \Delta_j - \Delta_j (v \cdot \nabla)$ the commutator between the transport operator $v\cdot\nabla$ and the frequency
localisation operator $\Delta_j$. 
Then we have
\begin{equation*}
\forall\, f \in \B\,, \qquad\qquad  2^{js}\left\| \big[ v \cdot \nabla, \Delta_j \big] f  \right\|_{L^\infty} \lesssim c_j \Big( \|\nabla v \|_{L^\infty} \| f \|_{\B} +
\|\nabla v \|_{B^{s-1}_{\infty, r}} \|\nabla f \|_{L^\infty} \Big)\,,
\end{equation*}
where $\big(c_j\big)_{j\geq -1}$ is a sequence belonging to the unit ball of $\ell^r$. 

\end{lemma}

The second commutator result deals with commutators between paraproduct operators and Fourier multipliers.
This essentially corresponds to Lemma 2.99 of \cite{BCD}.
From its proof in \cite{BCD}, it appears that the result holds regardless of the regularity of the symbol at $\xi=0$.

\begin{lemma}\label{l:ParaComm}
Let $\k$ be a smooth function on $\mathbb{R}^d\setminus\{0\}$, which is homogeneous of degree $m$ away from a neighborhood of $0$. Then, for a vector field $v$ such that $\nabla v \in L^\infty$, one has:
\begin{equation*}
\forall\, f \in \B\,, \qquad \left\| \big[ \mathcal{T}_v, \k(D) \big] f \right\|_{B^{s-m+1}_{\infty, r}}\, \lesssim\, \|\nabla v\|_{L^\infty} \|f\|_{\B}\,.
\end{equation*}
\end{lemma}


Additional commutator estimates, involving the Leray projection operators, are postponed to Subsection \ref{ss:Leray}.

\subsection{Transport equations in Besov spaces} \label{ss:transport}

In this section, we focus on transport equations in non-homogeneous Besov spaces. We refer to Chapter 3 of \cite{BCD} for a complete presentation of the subject.
We study the initial value problem
\begin{equation}\label{eq:TV}
\begin{cases}
\partial_t f + v \cdot \nabla f = g \\
f_{|t = 0} = f_0\,.
\end{cases}
\end{equation}
We will always assume the velocity field $v=v(t,x)$ to be a Lipschitz divergence-free function, \tsl{i.e.} $\D(v) = 0$.
It is therefore practical to formulate the following definition: the triplet $(s,p, r) \in \mathbb{R} \times [1, +\infty]^2$ is said to satisfy the Lipschitz condition if the inequality
\begin{equation}\label{i_eq:Lip}
s > 1 + \frac{d}{p} \qquad \text{or} \qquad s = \frac{d}{p} \text{ and } r = 1.
\end{equation}
holds. As we have explained above, this implies the embedding $B^s_{p,r} \hookrightarrow W^{1, \infty}$.

The main well-posedness result concerning problem \eqref{eq:TV} in Besov spaces is contained in the following statement, stated in the case
$p=+\infty$ (the only relevant one for our analysis).
We recall here that, when $X$ is Banach, the notation $C^0_w\big([0,T];X\big)$ refers to the space of functions which are continuous in time with values in $X$ endowed with its weak topology.
\begin{thm}\label{th:transport}
Let $(s, r) \in \mathbb{R} \times [1, +\infty]$ satisfy the Lipschitz condition \eqref{i_eq:Lip} with $p=+\infty$.
Given some $T>0$, let $g \in L^1_T(\B)$. Assume that $v \in L^1_T(\B)$ and that there exist real numbers
$q > 1$ and $M > 0$ for which $v \in L^q_T(B^{-M}_{\infty, \infty})$. Finally, let $f_0 \in \B$ be an initial datum.

Then, the transport equation \eqref{eq:TV} has a unique solution $f$ in:
\begin{itemize}
\item the space $C^0\big([0,T];\B\big)$, if $r < +\infty$;
\item the space $\left( \bigcap_{s'<s} C^0\big([0,T];B^{s'}_{\infty, \infty}\big) \right) \cap C^0_{w}\big([0,T];B^s_{p, \infty}\big)$, if $r = +\infty$.
\end{itemize}
Moreover, this unique solution satisfies the following estimate:
\begin{equation*} 
\| f \|_{L^\infty_T(\B)} \leq \exp \left( C\!\! \int_0^T \| \nabla v \|_{B^{s-1}_{\infty, r}} \right)
\left\{ \| f_0 \|_{\B} + \int_0^T \exp \left( - C\!\! \int_0^t \| \nabla v \|_{B^{s-1}_{\infty, r}} \right) \| g(t) \|_{\B} {\rm d} t  \right\},
\end{equation*}
for some constant $C = C(d, s, r)>0$.
\end{thm}


As discovered by Vishik \cite{Vis} and, with a different proof, by Hmidi and Keraani \cite{HK}, the previous statement can be improved when the Besov regularity index is $s=0$,
provided $\D(v) = 0$. Precisely, under these conditions,
the estimate in Theorem \ref{th:transport} can be replaced by an inequality which is linear with respect to $\|\nabla v\|_{L^1_T(L^\infty)}$.

\begin{thm}\label{th:AnnInnLinTV}
Assume that $\nabla v \in L^1_T(L^\infty)$ and that $v$ is divergence-free. Let $r \in [1, +\infty]$.
Then there exists a constant $C = C(d)$ such that, for any solution $f$ to problem \eqref{eq:TV} in $C^0\big([0,T];B^0_{\infty,r}\big)$, with the usual modification of $C^0$ into $C^0_w$
if $r=+\infty$, we have
\[ 
\| f \|_{L^\infty_T(B^0_{\infty, r})}\, \leq\, C\, \bigg\{ \| f_0 \|_{B^0_{\infty, r}}\, +\, \| g \|_{L^1_T(B^0_{\infty, r})} \bigg\}\;\left( 1+\int_0^T\| \nabla v(\tau) \|_{L^\infty}{\rm d} \tau \right)\,.
\] 
\end{thm}

\subsection{Leray projection} \label{ss:Leray}

This paragraph is concerned with the Leray projection operator $\P$, which is defined as
\[
\P\,:=\,\Id\,+\,\nabla(-\Delta)^{-1}\div
\]
in the sense of Fourier multipliers, namely
\begin{equation*}
\forall f \in \mc S, \qquad \what{(\P f)_j} (\xi) = \what{f_j}(\xi) -  \sum_{k} \frac{\xi_j \xi_k}{|\xi|^2} \what{f_k}(\xi).
\end{equation*}

The operator $\P$ can also be seen as a singular integral operator. Thus, Calder\'on-Zygmund theory may be applied to prove that $\P$ defines a bounded operator in the
$L^p \tend L^p$ topology, for any $1 < p < + \infty$.

In the endpoint space $L^\infty$, the Leray projector is no longer defined as a Fourier multiplier, because of the singularity of its symbol at $\xi = 0$.
However, for $j \in \{1, ..., d \}$, one may give sense of the operator $\Delta_{-1} \P \partial_j$ in $L^\infty$,
by using the integrability properties of the fundamental solution of the Laplacian. In fact, we have the following result, which corresponds to Proposition 8 of \cite{PP}.

\begin{prop} \label{p:TrucDesChinois}
Let $j \in \{ 1, ..., d \}$. We have a bounded operator
\begin{equation*}
\Delta_{-1} \P \, \partial_j : L^\infty \tend L^\infty.
\end{equation*}
\end{prop}

With this proposition at hand, we can prove the following statement. It is a commutator estimate between a transport operator and the Leray projector
in the critical Besov space $B^1_{\infty,1}$. It corresponds to Lemma 2.5 in \cite{Cobb-F_Rig}, although that result did not deal with the endpoint exponent $p = +\infty$.

\begin{lemma}\label{l:prodEstimates}
Let $f, g \in L^2 \cap B^1_{\infty, 1}$ be two vector fields such that $\div(f)\,=\,\div(g)\,=\,0$. The following inequality holds true:
\begin{equation*}
\left\| \big[ f \cdot \nabla, \P \big] g \right\|_{B^1_{\infty, 1}} \lesssim \left\| f \right\|_{B^1_{\infty, 1}} \left\| g \right\|_{B^1_{\infty, 1}}.
\end{equation*}
\end{lemma}

\begin{proof}
We start by noticing that, thanks to the regularity assumption on both $f$ and $g$, the commutator is well-defined, since $\P$ always acts on $L^2$ vector fields. This also guarantees
us that $\P g=g$.
Using this fact, we can write the Bony decomposition for the products involved in the commutator: we get
\begin{align*}
\big[ f \cdot \nabla, \P \big] g & = \sum_{k=1}^d\Big(\big[ \mc T_{f_k} , \P \big] \partial_k g + \mc T_{\P \partial_k g} (f_k)
- \P \mc T_{\partial_k g} (f_k) + \mc R (\P \partial_k g, f_k) - \P \mc R (\partial_k g, f_k)\Big) \\
& = \sum_{k=1}^d \Big( \big[ \mc T_{f_k} , \P \big] \partial_k g + \mc T_{\partial_k g} (f_k) - \P \partial_k \mc T_{g} (f_k) + \mc R (\partial_k g, f_k) - \P \partial_k \mc R (g, f_k)\Big)\,.
\end{align*}
The first, second and fourth summand are easy to bound, by using Lemma \ref{l:ParaComm} and Proposition \ref{p:op}.
The only new difficulty comes from the third and fifth terms, which involve the Leray projection.
However, since the operator $\P$ is always written in composition with a derivative $\partial_k$ in those terms, we may use Proposition \ref{p:TrucDesChinois} to write
\begin{equation*}
\begin{split}
\left\| \P \partial_k \mc T_{g} (f_k) \right\|_{B^1_{\infty, 1}} & \leq \left\| \Delta_{-1} \P \partial_k \mc T_{g} (f_k) \right\|_{L^\infty} + \sum_{m \geq 0} 2^m \left\| \Delta_m \P \mc T_{\partial_k g} (f_k) \right\|_{L^\infty} \\
& \lesssim \| \mc T_g (f) \|_{L^\infty} + \sum_{m \geq 0} 2^m \left\| \Delta_m \P \mc T_{\partial_k g} (f_k) \right\|_{L^\infty}.
\end{split}
\end{equation*}
Next, Lemma \ref{l:homMult} implies that, for $m \geq0$, the operator $\Delta_m \P$ is bounded on $L^\infty$. Thus, we may use Proposition \ref{p:op} to obtain
\[
2^m \left\| \Delta_m \P \mc T_{\partial_k g} (f_k) \right\|_{L^\infty}\,\lesssim\,\| f \|_{B^1_{\infty, 1}} \| g \|_{B^1_{\infty, 1}}\,c_m\,,
\]
for a suitable sequence $\big(c_m\big)_{m\in\N}\,\in\,\ell^1$ of unitary norm. In the end, we get the estimate
\begin{equation*}
\left\| \P \partial_k \mc T_{g} (f_k) \right\|_{B^1_{\infty, 1}} \lesssim \| f \|_{B^1_{\infty, 1}} \| g \|_{B^1_{\infty, 1}}\,.
\end{equation*}
For the fifth term $\P \partial_k \mc R (g, f_k)$, we can proceed in a similar way. This completes the proof of the sought bounds for the commutator.
\end{proof}

From the previous lemma, we immediately deduce the next result.
\begin{cor}\label{c:prodEstimates2}
Let $f\in L^2 \cap B^1_{\infty, 1}$ and $v \in L^2 \cap B^2_{\infty, 1}$ be two divergence-free vector fields. The following inequality holds true:
\begin{equation*}
\left\| \P (f \cdot \nabla) v \right\|_{B^1_{\infty, 1}} \lesssim  \left\| f \right\|_{B^1_{\infty, 1}} \left\| v \right\|_{B^2_{\infty, 1}}.
\end{equation*}
\end{cor}

\begin{proof}
For proving the previous statement, it is enough to write
\[
\P (f \cdot \nabla) v\,=\,(f \cdot \nabla) v\,-\,\big[f\cdot\nabla, \P\big]v\,,
\]
where we have used also the fact that $\div(v)=0$.
The first term in the right-hand side can be estimates directly, whereas we use the bounds of Lemma \ref{l:prodEstimates} for the second one.
%
\end{proof}

\section{Continuation criteria} \label{s:cont-crit}

In this section, we state and prove our main results concerning continuation criteria for solutions of the ideal MHD equations \eqref{ieq:mhd}. In Subsection \ref{ss:cont-u},
we focus on a criterion based only on the velocity field, while in Subsection \ref{ss:cont-els} we will present a continuation criterion in terms of the Els\"asser variables.

\subsection{A continuation criterion based on the velocity only} \label{ss:cont-u}

The main result of this section is the following statement.

\begin{thm}
Let $(u_0, b_0) \in L^2(\R^d) \cap B^1_{\infty, 1}(\R^d)$ be a set of divergence-free initial data. Consider $T > 0$ such that the ideal MHD system, supplemented with
those initial data, has a unique solution $(u, b)$ in the space $C^0\big([0,T[\,;L^2(\R^d)\,\cap\,B^1_{\infty, 1}(\R^d)\big)$.

Then this solution may be continued beyong the time $T$ provided that 
\begin{equation}\label{eq:contCritU}
\int_0^T \left\| \nabla^2 u(t) \right\|_{L^\infty} \dt < + \infty.
\end{equation}
\end{thm}

\begin{proof}
As far as continuation results go (keep in mind Theorem \ref{t_i:WP}), we already know that the solution may be prolonged beyond time $T$ if we have 
\begin{equation}\label{eq:contProofEQ2}
\int_0^T \Big( \| \nabla u \|_{L^\infty} + \| \nabla b \|_{L^\infty} \Big) \dt < +\infty.
\end{equation}
Therefore, we only have to show that the integral in \eqref{eq:contProofEQ2} is finite under condition \eqref{eq:contCritU}.

We start by recalling that, by a simple energy method, we have
\begin{equation} \label{est:energy}
\sup_{t\in[0,T[}\Big(\left\|u(t)\right\|_{L^2}\,+\,\left\|b(t)\right\|_{L^2}\Big) \lesssim \left\|u_0\right\|_{L^2}\,+\,\left\|b_0\right\|_{L^2}.
\end{equation}
By using this bound together with the Bernstein inequalities, we can estimate
\begin{align}
\| \nabla u \|_{L^\infty} & \leq \| \Delta_{-1} \nabla u \|_{L^\infty} + \sum_{m \geq 0} \| \Delta_m \nabla u \|_{L^\infty}\,\lesssim 
\| u \|_{L^2} + \sum_{m \geq 0} 2^{-m} \| \Delta_m \nabla^2 u \|_{L^\infty}  \label{est:Du} \\
&  \lesssim \left\|\big( u_0, b_0\big) \right\|_{L^2} + \| \nabla^2 u \|_{L^\infty}. \nonumber
\end{align}
Thus, under condition \eqref{eq:contCritU} we deduce that $\| \nabla u \|_{L^1_T(L^\infty)}\,\lesssim\,T\,\left\| \big( u_0, b_0\big) \right\|_{L^2} + \| \nabla^2 u \|_{L^1_T(L^\infty)}\,<\,+\infty$.

It remains us to show that also $\| \nabla b \|_{L^1_T(L^\infty)}$ is finite.
This will be a consequence of the fact that $b$ solves a linear transport equation, which we may differentiate to obtain estimates on the first derivative $\nabla b$. Precisely, for $j = 1, 2$, we have
\begin{equation*}
\partial_t \partial_j b + (u \cdot \nabla) \partial_j b = - ( \partial_j u \cdot \nabla) b + (\partial_j b \cdot \nabla) u + (b \cdot \nabla) \partial_j u.
\end{equation*}
A basic $L^\infty$-estimate immediately gives, for all $0\leq t<T$, the bound
\begin{equation}\label{eq:contProofEQ1}
\| \nabla b(t) \|_{L^\infty} \leq \| \nabla b_0 \|_{L^\infty} + \int_0^t \Big\{ \| \nabla u \|_{L^\infty} \| \nabla b \|_{L^\infty} + \| b \|_{L^\infty} \| \nabla^2 u \|_{L^\infty} \Big\} {\rm d}\tau
\end{equation}
The term $\| \nabla u \|_{L^\infty}$ has already been estimated in \eqref{est:Du}. Arguing similarly, we can find an upper bound for $\| b \|_{L^\infty}$
which involve only the quantities we have at our disposal: by separating low and high frequencies, we have
\begin{align*}
\| b \|_{L^\infty} & \leq \| \Delta_{-1} b \|_{L^\infty} + \sum_{m \geq 0} \| \Delta_m b \|_{L^\infty}\, \lesssim \| b \|_{L^2} + \sum_{m \geq 0} 2^{-m} \| \Delta_m \nabla b \|_{L^\infty} \\
& \lesssim \left\| \big( u_0, b_0\big)\right\|_{L^2} + \| \nabla b \|_{L^\infty}.
\end{align*}
Plugging \eqref{est:Du} and the previous bound into \eqref{eq:contProofEQ1}, we obtain an integral inequality which is linear with respect to $\| \nabla b \|_{L^\infty}$: for
all $0\leq t<T$, we have
\begin{align*}
\sup_{\tau\in[0,t]}\| \nabla b(\tau) \|_{L^\infty} &\lesssim \| \nabla b_0 \|_{L^\infty} + \left\| \big(u_0, b_0\big) \right\|_{L^2} \int_0^t \| \nabla^2 u \|_{L^\infty} {\rm d}\t \\
&\qquad\qquad\qquad
+ \int_0^t \| \nabla b \|_{L^\infty} \Big( \left\| \big(u_0, b_0\big) \right\|_{L^2} + \| \nabla^2 u \|_{L^\infty} \Big) {\rm d}\t.
\end{align*}
By using Gr\"onwall's lemma, we deduce that $\| \nabla b \|_{L^\infty}$ must be bounded as long as condition \eqref{eq:contCritU} is fulfilled. So the integral \eqref{eq:contProofEQ2}
must therefore also be finite while \eqref{eq:contCritU} holds.
\end{proof}

We conclude this part with a remark concerning a continuation criterion in terms of the magnetic field only.
\begin{rmk}
The evolution of the velocity field is dictated by the momentum equation, which is quadratic with respect to $u$.
This implies that it is not possible, in general, to find the kind of linear estimates that would yield a continuation criterion dispensing of any condition on the velocity field.
However, in the special case of space dimension $d = 2$, the vorticity equation is linear in $u$: we have
\begin{equation*}
\partial_t \omega + u \cdot \nabla \omega = b \cdot \nabla j\,.
\end{equation*}
Therefore, in that case it is possible to bound $u$, or $\omega$, given good enough bounds on the magnetic field. This leads to a continuation criterion based only on $b$:
we have $T < T^*$ as long as
\begin{equation*}
\int_0^T \| j \|^2_{B^1_{\infty, 1}} \dt < + \infty.
\end{equation*}
\end{rmk}

\subsection{A continuation criterion based on the Els\"asser variables} \label{ss:cont-els}
As we have explained in the introduction, the magnetic field is not the only quantity which solves a linear equation: the Els\"asser variables also do so.
This means that we are able to find a continuation criterion based only on either $\al = u+ b$ or $\bt = u-b$.

\begin{thm}\label{t:contElsUplusB}
Let $(u_0, b_0) \in L^2(\R^d) \cap B^1_{\infty, 1}(\R^d)$ be a set of divergence-free initial data. Consider $T > 0$ such that the ideal MHD system, supplemented with
those initial data, has a unique solution $(u, b)$ in the space $C^0\big([0,T[\,;L^2(\R^d)\,\cap\,B^1_{\infty, 1}(\R^d)\big)$. 

Then, if we denote $\omega=\curl(u)$ and $j=\curl(b)$, one has
\begin{equation}\label{eq:omega+-j}
\int_0^T \big\| \, \omega + j \, \big\|_{B^0_{\infty, 1}} \dt < + \infty\qquad \Longleftrightarrow \qquad 
\int_0^T \big\| \, \omega - j \, \big\|_{B^0_{\infty, 1}} \dt < + \infty\,.
\end{equation}
In addition, in the case those integrals are finite, the solution $(u,b)$ may be continued beyond $T$ into a solution belonging to the same regularity class.
\end{thm}


\begin{proof}
We already know, from Proposition 5.7 in \cite{CF3}, that the solution may be continued beyond the time $T$ if and only if
\begin{equation} \label{est:cc-Els}
\int_0^T \Big\{ \| \nabla(u + b) \|_{L^\infty} + \| \nabla(u - b) \|_{L^\infty} \Big\} \dt < +\infty.
\end{equation}
Throughout this proof, we assume that
\begin{equation}\label{eq:contElsUplusB}
\int_0^T \big\| \, \omega + j \, \big\|_{B^0_{\infty, 1}} \dt < + \infty.
\end{equation}
We are going to show that this condition is enough to ensure that also
\begin{equation} \label{eq:contU-B}
\int_0^T \big\| \, \omega - j \, \big\|_{B^0_{\infty, 1}} \dt < + \infty\,,
\end{equation}
and that the finiteness of both integrals implies \eqref{est:cc-Els}.

The exact same argument will apply also when considering the quantity $\omega-j$, whence the claimed equivalence.

\medbreak
We start the proof by remarking that, splitting into low and high frequencies as done in \eqref{est:Du} and using the Biot-Savart law
\[
f_k\,=\,(-\Delta)^{-1}\sum_{j=1}^d\d_j\big[\curl (f)\big]_{jk}\,,
\]
it is easy to show that, for any divergence-free vector field $f$, one has
\begin{equation} \label{est:low-high-f}
\left\|\nabla f\right\|_{L^\infty}\,\leq\,\|f\|_{B^{1}_{\infty,1}}\,\lesssim\,\|f\|_{L^2}\,+\,\|\curl (f)\|_{B^{0}_{\infty,1}}\,.
\end{equation} 
In the above equations, we have denoted $\curl (f)$ the matrix such that $\big[\curl (f)\big]_{jk}\,=\,\d_jf_k\,-\,\d_kf_j$, with the usual identification $\curl (f) = \nabla\times f$ in dimension $d=3$, and
$\curl (f)\,=\,\d_1f_2-\d_2f_1$ in dimension $d=2$. Thus, after noticing that
\[
\left\|\al\right\|_{L^2}\,\leq\,\left\|\big(u,b\big)\right\|_{L^2}\,\leq\,\left\|\big(u_0,b_0\big)\right\|_{L^2}
\]
in view of the energy inequality \eqref{est:energy}, we get
\begin{equation} \label{est:low-high}
\| \nabla (u + b) \|_{L^\infty}\,\leq\,\| u + b \|_{B^1_{\infty,1}}\,\lesssim\,\left\|\big(u_0,b_0\big)\right\|_{L^2}\,+\,\left\|\omega+j\right\|_{B^0_{\infty,1}}\,.
\end{equation} 
So, it remains us to show that the integral in \eqref{eq:contU-B} remains finite under condition \eqref{eq:contElsUplusB}. For this, we note that $\bt = u - b$ solves the
\emph{linear} equation
\begin{equation} \label{eq:beta}
\partial_t \bt + (\al \cdot \nabla) \bt = \big[ \al \cdot \nabla, \P \big] \bt\,,
\end{equation}
where $\P$ is the Leray projector onto the space of divergence-free vector fields, as introduced in Subsection \ref{ss:Leray}.

By applying the dyadic block $\Delta_j$, for $j \geq -1$, to equation \eqref{eq:beta}, we obtain
\begin{equation}\label{eq:deltaJEstimate}
\Big( \partial_t + \al \cdot \nabla \Big) \Delta_j \bt = \Delta_j \big[ \al \cdot \nabla, \P \big] \bt + \big[ \al \cdot \nabla, \Delta_j \big] \bt.
\end{equation}
The second commutator in the right-hand side of this equation can be estimated with the help of Lemma \ref{l:CommBCD}, while we can resort to Lemma \ref{l:prodEstimates} for bounding
the first one.
Then, we produce $L^\infty$ estimates for the dyadic blocks $\Delta_j \bt$ appearing in the transport equation \eqref{eq:deltaJEstimate}: we have
\begin{equation*}
2^j \| \Delta_j \bt \|_{L^\infty} \lesssim 2^j \| \Delta_j \bt_0 \|_{L^\infty} + \int_0^T c_j(t) \| \al \|_{B^1_{\infty, 1}} \| \bt \|_{B^1_{\infty, 1}} \dt\,,
\end{equation*}
for suitable sequences $\big(c_j(t)\big)_{j\geq -1}$ belonging to the unit sphere of $\ell^1$.
By summing the previous inequality over all $j \geq 1$, we get
\begin{equation*}
\| \bt \|_{{L^\infty_T}(B^1_{\infty, 1})} \lesssim \| \bt_0 \|_{B^1_{\infty, 1}} + \int_0^T \| \al \|_{B^1_{\infty, 1}} \| \bt \|_{B^1_{\infty, 1}} \dt.
\end{equation*}
At this point, we estimate the $B^1_{\infty, 1}$ norm of $\al = u + b$ by using \eqref{est:low-high}, and
we finally infer an integral inequality which is linear with respect to $\| \bt \|_{B^1_{\infty, 1}}$:
\begin{equation*}
\| \bt \|_{L^\infty_T (B^1_{\infty, 1})} \lesssim \| \bt_0 \|_{B^1_{\infty,1}} + \int_0^T \Big( \| \al_0 \|_{L^2} + \| \, \omega + j \, \|_{B^0_{\infty, 1}} \Big) \| \bt \|_{B^1_{\infty, 1}} \dt\,.
\end{equation*}
Thus, we may use Gr\"onwall's lemma to end the proof.
\end{proof}

\begin{rmk}
Because of the presence of the pressure terms in the equations and of the fact that we work in a critical regularity framework,
we are unable to obtain a continuation criterion depending only on the $L^\infty$ norm of $\nabla (u\pm b)$.

However, in a subcritical regularity framework $B^s_{\infty,r}$ with $s>1$, we believe that the same method of \cite{B-K-M} applies to give a continuation criterion
in terms of the finiteness of the norm $\|\curl(u\pm b)\|_{L^1_T(L^\infty)}$. We do not treat the extension of our result in this direction here.
\end{rmk}

\section{Improved lifespan for planar flows} \label{s:lifespan}

In this section, we present a refinement of Theorem \ref{t_i:lifespan} of the introduction. As in that result, we exhibit a lower bound for the lifespan of the solutions in space dimension
$d=2$, which implies that the lifespan tends to $+\infty$ when the size of the magnetic field tends to $0$. The point is that we require higher regularity on the initial velocity field only.

\subsection{Statement of the result} \label{ss:statement}

We present here the precise statement concerning the improved lower bound for the lifespan of the solutions in two space dimensions.

\begin{thm}\label{t:lifespan}
Let $(u_0, b_0)\,\in\,L^2(\R^2)$ be a set of divergence-free initial data such that $u_0 \in B^2_{\infty, 1}(\R^2)$ and $b_0\in B^1_{\infty, 1}(\R^2)$.

Then, the lifespan $T>0$ of the corresponding solution $(u, b)\in C\big([0,T[\,;L^2(\R^2)\cap B^1_{\infty,1}(\R^2)\big)$ of the $2$-D ideal MHD system \eqref{ieq:mhd}, given by Theorem \ref{t_i:WP},
enjoys the following lower bound:
\begin{equation}\label{eq:TstarInequality}
T \geq \frac{C}{\| u_0 \|_{L^2 \cap B^2_{\infty, 1}}} \log \left\{ 1 + C\,\log \left[ 1 + C\,\log \left( 1 + C \frac{\| u_0 \|_{L^2 \cap B^2_{\infty, 1}} }{\| b_0 \|_{B^1_{\infty, 1}}} \right) \right] \right\},
\end{equation}
where $C > 0$ is a constant independent of the initial data.
\end{thm}

Before proving the previous statement, a couple of remarks are in order.

\begin{rmk}
Our result is stated in two dimensions of space. This is crucial, as the proof relies on the existence of global solutions to the Euler system.
However, this is the only point in our argument that is specific to $d = 2$, and our proof may be adapted to all dimensions $d \geq 2$ to show the following fact:
if we denote by $T_E > 0$ the lifespan of the unique $B^2_{\infty, 1}$ solution $v$ of the Euler problem with initial datum $v_{|t=0} = u_0$, and if we set
$\| b_0 \|_{B^1_{\infty, 1}} = \veps$, then the lifespan $T_\veps$ of the solution $(u, b)$ to the ideal MHD system \eqref{ieq:mhd} satisfies
\begin{equation*}
T_\veps\, \tend\, T_E \qquad\qquad \mbox{ as } \quad \veps\, \rightarrow\, 0^+\,.
\end{equation*}
\end{rmk}

\begin{rmk}
For proving Theorem \ref{t:lifespan}, we use the fact that $(u, b)$ is a finite energy solution only to recast the ideal MHD system into the Els\"asser variables.
However, at the quantitative level, the magnetic energy $\| b_0 \|_{L^2}$ plays absolutely no role in the computations.

Thus, our approach, and the result, can be adapted to other situations where the solutions have infinite energy, but one can use the equivalence of the original ideal MHD system with 
its projected counterpart (\tsl{i.e.} the system obtained after application of the Leray projector $\P$ to the equations) and with the Els\"asser formulation.
We refer to \cite{Cobb} for more details on that topic.
\end{rmk}

\subsection{Proof of the improved lower bound for the lifespan} \label{ss:proof}

This subsection is devoted to the proof of Theorem \ref{t:lifespan}. The main idea of our method is to compare the ideal MHD system to the classical homogeneous Euler system, which is known to be
globally well-posed in $\R^2$ in our functional framework.

We divide the proof into three steps.

\paragraph{Step 1: solving the Euler equations.}
To begin with, we solve the incompressible Euler equations with initial datum $u_0$. More precisely, let
$v : \mathbb{R} \times \mathbb{R}^2 \tend \mathbb{R}^2$ be the unique global solution of the initial value problem
\begin{equation}\label{eq:Euler}
\begin{cases}
\partial_t v + (v \cdot \nabla) v + \nabla p = 0 \\[1ex]
\D(v) = 0 \\[1ex]
v_{|t=0} = u_0
\end{cases}
\end{equation}
which lies in the class $C^0\big(\R_+;L^2(\R^2) \cap B^2_{\infty, 1}(\R^2)\big)$. We refer \tsl{e.g.} to Chapter 7 of \cite{BCD} for details.

In the rest of the proof, we need the following lemma. Even though the estimate contained therein is well-known, we were not able to find a precise reference for it. Therefore,
we also provide a full proof.

\begin{lemma}\label{l:order2forVelovity}
Set $V_0 = \| u_0 \|_{L^2 \cap B^2_{\infty, 1}}$. Then, for all $T > 0$, we have the following inequality:
\begin{equation*}
\sup_{t\in[0,T]}\left\|v(t)\right\|_{L^2\cap B^2_{\infty,1}}\, \leq\, C\, V_0 \,\exp \left(C\, T\, V_0\, e^{C\, T\, V_0} \right)\,,
\end{equation*}
for some numerical constant $ C > 0$, independent of $u_0$.
\end{lemma}

\begin{proof}
Since the following energy conservation holds true, namely
\begin{equation} \label{est:Euler-en}
\forall\,t\geq0\,,\qquad\qquad \|v(t)\|_{L^2}\,\leq\,\|v_0\|_{L^2}\,,
\end{equation}
we only have to bound the Besov norm of $v$.
For this, we resort to the vorticity form of the Euler equations. Define the vorticity $\Omega = \partial_1 v_2 - \partial_2 v_1$ of the flow $v$, which can be recovered
from $\Omega$ by the $2$-D Biot-Savart law $v = - \nabla^\perp (- \Delta)^{-1} \Omega$. Then, $\Omega$ solves the pure transport equation
\begin{equation}\label{eq:vorticity}
\partial_t \Omega + v \cdot \nabla \Omega = 0\,,\qquad\qquad\mbox{ with }\qquad \Omega_{|t=0} = \Omega_0 := \d_1v_{0,2} - \d_2v_{0,1}.
\end{equation}

First of all, we find $B^0_{\infty, 1}$ bounds for $\Omega$, by using the linear transport estimates of Theorem \ref{th:AnnInnLinTV}: we get, for any $T>0$, the bound
\begin{equation}\label{eq:vorticityInequality}
\| \Omega \|_{{L^\infty_T}(B^0_{\infty, 1})} \lesssim \| \Omega_0 \|_{B^0_{\infty, 1}} \left( 1 + \int_0^T \| \nabla v \|_{L^\infty} \dt \right).
\end{equation}
To control the norm of the gradient $\nabla v$ appearing in this estimate, we resort to the inequality exhibited in \eqref{est:low-high-f}:
by combining the latter with \eqref{est:Euler-en}, 
and then using Gr\"onwall's lemma, from \eqref{eq:vorticityInequality} we obtain
\begin{equation*}
\| \Omega \|_{B^0_{\infty, 1}} \lesssim \| \Omega_0 \|_{B^0_{\infty, 1}} \Big( 1 + T \| u_0 \|_{L^2} \Big) \exp \left( c T \| \Omega_0 \|_{B^0_{\infty, 1}} \right)\,.
\end{equation*}
So, by adding the energy $\| v \|_{L^2}$ to both sides of this inequality, we obtain a first order estimate of the solution $v$, namely
\begin{equation*}
\begin{split}
\| v \|_{L^\infty_T (B^1_{\infty, 1})} & \lesssim \| u_0 \|_{L^2} + \| \Omega_0 \|_{B^0_{\infty, 1}} \Big( 1 + T \| u_0 \|_{L^2} \Big) \exp \left( c T \| \Omega_0 \|_{B^0_{\infty, 1}} \right) \\
& \lesssim V_0 (1 + T V_0) e^{cTV_0} \lesssim V_0 e^{cT V_0}\,.
\end{split}
\end{equation*}

Next, we differentiate the vorticity equation \eqref{eq:vorticity} to obtain a second order bound on $v$. For $k = 1, 2$, we get
\begin{equation*}
\partial_t \partial_k \Omega + u \cdot \nabla \partial_k \Omega = - \partial_k v \cdot \nabla \Omega\,.
\end{equation*}
Applying the linear estimate of Theorem \ref{th:AnnInnLinTV} one more time, we find
\begin{equation*}
\| \nabla \Omega \|_{{L^\infty_T}(B^0_{\infty, 1})} \lesssim \left\{ \| \nabla \Omega_0 \|_{B^0_{\infty, 1}} + \sum_{k=1,2 }
\int_0^T \| \partial_k v \cdot \nabla \Omega \|_{B^0_{\infty, 1}} \dt \right\} \left( 1 + \int_0^T \| \nabla v \|_{L^\infty} \dt \right).
\end{equation*}
The only difficulty in using the previous inequality is that the space $B^0_{\infty, 1}$ is not an algebra. Therefore we must be careful with the first integral,
which involves the second order derivative $\nabla \Omega$. However, using the Bony decomposition and Proposition \ref{p:op}, we find that the function product is a continuous map in
the $B^1_{\infty, 1} \times B^0_{\infty, 1} \tend B^0_{\infty, 1}$ topology. Therefore, by adding the energy $\| v \|_{L^2}$ to both sides of the previous inequality,
we find an integral estimate for the quantity $\phi(t)\,:=\,\left\|v\right\|_{L^\infty_t(L^2\cap B^2_{\infty,1})}$, namely
\begin{equation*}
\phi(T) \lesssim \left\{ V_0 + \int_0^T \phi(t) V_0 e^{ct V_0} \dt \right\} \left( 1 + \int_0^T V_0 e^{ct V_0} \dt \right) \lesssim V_0 e^{cT V_0} \left\{ 1 + \int_0^T \phi(t) \dt \right\}.
\end{equation*}
An application of Gr\"onwall's lemma to this last inequality ends the proof.
\end{proof}

\paragraph{Step 2: using the Els\"asser formulation.}
This having been established, we come back to the ideal MHD system \eqref{ieq:mhd}. Thanks to the finite energy condition on the solution $(u,b)$,
we can recast the equations in Els\"asser variables, as done in Subsection \ref{ss:els}, and compare the Els\"asser system \eqref{eq:els} to the homogeneous Euler equations \eqref{eq:Euler}.

In order to do so, we start by defining the difference functions
\begin{equation*}
\delta \al = \al - v = (u-v) + b \qquad \quad \mbox{ and } \qquad\quad \delta \bt = \bt - v = (u-v) - b\,.
\end{equation*}
By using equations \eqref{eq:els} and \eqref{eq:Euler}, we find that the couple $(\de\al,\de\bt)$ satisfies the system
\begin{equation}\label{eq:dadbSyst}
\begin{cases}
\partial_t (\delta \al) + ( \bt \cdot \nabla) \delta \al + (\delta \bt \cdot \nabla) v + \nabla (\delta \pi_1) = 0 \\
\partial_t (\delta \bt) + ( \al \cdot \nabla) \delta \bt + (\delta \al \cdot \nabla) v + \nabla (\delta \pi_2) = 0 \\
\D (\delta \al) = \D (\delta \bt) = 0\,,
\end{cases}
\end{equation}
where, for $k \in \{ 1, 2 \}$, we have set $\delta \pi_k = \pi_k - p$  as being the differences of the pressure functions.

The difference functions $\delta \al$ and $\delta \bt$ are very well suited for our purposes, as their initial values depend only on the initial magnetic field:
\begin{equation*}
\delta \al (0) = b_0 \qquad\quad \text{ and } \qquad\quad \delta \bt (0) = - b_0.
\end{equation*}
In addition, estimating $\delta \alpha$ and $\delta \bt$ will immediately provide control for the solution $(\al, \bt)$ of the Els\"asser system \eqref{eq:els}, thanks to
explicit bounds for the regular solution $v$.
With that in mind, we seek to estimate the $B^1_{\infty, 1}$ norms of $(\delta \al, \delta \bt)$. To do this, we see from \eqref{eq:dadbSyst} that we will need $B^2_{\infty, 1}$ bounds on $v$,
which are given in Lemma \ref{l:order2forVelovity} above.

Let $j \geq -1$. By applying the Leray projection operator $\P$, followed by the dyadic block $\Delta_j$, to the first two equations in \eqref{eq:dadbSyst}, we obtain
\begin{equation}\label{eq:dyadicDaDb}
\begin{cases}
\partial_t \Delta_j (\delta \al) + (\bt \cdot \nabla) \Delta_j (\delta \al) + \Delta_j \P (\delta \bt \cdot \nabla) v =
\Delta_j \big[ \bt \cdot \nabla, \P \big] \delta \al + \big[ \bt \cdot \nabla, \Delta_j \big] \delta \al \\[1ex]
\partial_t \Delta_j (\delta \bt) + (\al \cdot \nabla) \Delta_j (\delta \bt) + \Delta_j \P (\delta \al \cdot \nabla) v =
\Delta_j \big[ \al \cdot \nabla, \P \big] \delta \bt + \big[ \al \cdot \nabla, \Delta_j \big] \delta \bt\,.
\end{cases}
\end{equation}
We now perform $L^\infty$ estimates on that system. 
By using the commutator estimates of Lemmas \ref{l:CommBCD} and \ref{l:prodEstimates} and the bounds of Corollary \ref{c:prodEstimates2},
we get
\begin{equation*}
2^j \|\Delta_j (\delta \al, \delta \bt) \|_{L^\infty} \lesssim 2^j \| \Delta_j b_0 \|_{L^\infty} + \int_0^t c_j(\tau) \|(\delta \al, \delta \bt) \|_{B^1_{\infty, 1}} \Big( \| v \|_{B^2_{\infty, 1}} + \| (\al, \bt) \|_{B^1_{\infty, 1}} \Big) {\rm d} \tau,
\end{equation*}
where $\big( c_j(\t) \big)_{j \geq -1}$ are sequences all belonging to the unit sphere of $\ell^1$. By applying the Minkowski inequality to this last inequality, we obtain
\begin{equation}\label{eq:intInequality1}
\| (\delta \al, \delta \bt) \|_{B^1_{\infty, 1}} \lesssim \| b_0 \|_{B^1_{\infty, 1}} + \int_0^t \| (\delta \al, \delta \bt) \|_{B^1_{\infty, 1}} \Big( \| v \|_{B^2_{\infty, 1}} +
\| (\al, \bt) \|_{B^1_{\infty, 1}} \Big) {\rm d} \tau\,.
\end{equation}

\paragraph{Step 3: final estimates.}
With the previous estimate \eqref{eq:intInequality1} at hand, we can conclude the proof of Theorem \ref{t:lifespan}.
In order to simplify the next computations, we define, for all $T > 0$, the quantities
\begin{equation*}
E(T) = \sup_{t \in [0, T]} \big\| (\delta \al (t), \delta \bt (t)) \big\|_{B^1_{\infty, 1}} \qquad \mbox{ and } \qquad \phi(T) = \sup_{t \in [0, T]} \| v(t) \|_{B^2_{\infty, 1}}\,.
\end{equation*}

For completing the proof, it remains us to estimate the $B^1_{\infty, 1}$ norm of the solution $(\al, \bt)$ and find an inequality for $E(t)$, since $\phi(t)$ is finite at every time $t > 0$.
Now, we observe that $\al = \delta \al + v$ and $\bt = \de\bt + v$, which implies
\[
 \left\| \big(\al(t), \bt(t)\big) \right\|_{B^1_{\infty, 1}}\,\leq\,E(t)\,+\,\phi(t)\,.
\]
Using this estimate in \eqref{eq:intInequality1}, we infer the integral inequality
\begin{equation*}
E(t) \lesssim E_0 + \int_0^t E(\tau) \Big( E(\tau) + \phi (\tau) \Big) {\rm d} \tau = E_0 + \int_0^t E^2(\tau) {\rm d} \tau + \int_0^t E(\tau) \phi(\tau) {\rm d} \tau.
\end{equation*}
To get rid of the linear part in this inequality, namely the last summand in the right-hand side, we start by using Gr\"onwall's lemma. Thus we get, for all $T > 0$,
the bound
\begin{equation*}
E(T) \lesssim \left( E_0 + \int_0^T E^2(t) {\rm d} t \right) e^{c T \phi (T)}\,,
\end{equation*}
where $c > 0$ is some irrelevant numerical constant. Next, in order to bound $E(T)$ on some time interval, we define the time $T^* > 0$ as
\begin{equation*}
T^* := \sup \left\{ T > 0\,\Big| \quad \int_0^T E^2(t) {\rm d} t \leq E_0 \right\}.
\end{equation*}
So, for all times $0 \leq T \leq T^*$, we must have the bound $E(T) \lesssim E_0 e^{c T \phi(T)}$. Therefore, by definition of $T^*$, we must have the inequality
\begin{equation*}
 e^{2c T^* \phi (T^*)}\, T^*\, \geq\, \frac{1}{E_0}\,.
\end{equation*}
This inequality alone proves that the time $T^*$ on which the solution is known to satisfy uniform $B^1_{\infty, 1}$ estimates is arbitrarily large if $E_0$ is made as small as necessary.
However, to find more quantitative inequalities, we need to use the upper bound for $\phi (T)$ provided by Lemma \ref{l:order2forVelovity}: we find that, for all $T \leq T^*$, one has
\begin{equation*}
\int_0^T E^2(t) \dt \lesssim \frac{E_0^2}{V_0} T V_0 \exp \Big\{c TV_0 \exp \left(c T V_0 e^{c T V_0} \right) \Big\}.
\end{equation*}
Thus, by definition of $T^*$, for $T=T^*$ we must have
\[
 \frac{E_0^2}{V_0} T^* V_0 \exp \Big\{c T^*V_0 \exp \left(c T^* V_0 e^{c T^* V_0} \right) \Big\}\,\geq\,E_0\,.
\]
By using the inequality $x \leq e^{x} - 1$ in the previous estimate and applying the logarithm function three times, we prove the theorem.



\end{document}